\documentclass[12pt]{amsart}
\usepackage{amscd}
\usepackage{amssymb}
\usepackage{mathrsfs}
\usepackage[all]{xy}
\usepackage{color,fullpage}
\usepackage[OT2,T1]{fontenc}
\usepackage{enumerate}
\usepackage{hyperref}

\DeclareSymbolFont{cyrletters}{OT2}{wncyr}{m}{n}
\DeclareMathSymbol{\Sha}{\mathalpha}{cyrletters}{"58}

\newcommand{\ms}{\mathscr}
\newcommand{\wh}{\widehat}
\newcommand{\frf}{\operatorname{frac}}
\newcommand{\mbb}{\mathbb}
\newcommand{\im}{\operatorname{im}}

\newcommand{\Q}{\mathbb{Q}}
\newcommand{\Z}{\mathbb{Z}}
\newcommand{\oper}[1]{\operatorname{#1}}
\newcommand{\Spec}{\oper{Spec}}
\newcommand{\Gal}{\oper{Gal}}
\newcommand{\cha}{\oper{char}}
\newcommand{\Supp}{\oper{Supp}}
\newcommand{\Hom}{\oper{Hom}}
\newcommand{\cd}{\oper{cd}}
\newcommand{\td}{\oper{tr.deg.}}
\newcommand{\nr}{\oper{nr}}

\newcommand{\til}[1]{\widetilde{#1}}

\newcommand{\Br}{\operatorname{Br}}

\newcommand{\nc}{\newcommand}
\nc{\flag}[1]{{\color{black}#1}}
\nc{\dvf}{{\flag K}} 
\nc{\dvr}{{\flag{\ms O_K}}} 
\nc{\sgf}{{\flag F}} 
\nc{\fy}{{\flag E}} 
\nc{\sgm}{\flag{\ms X}} 
\nc{\my}{\flag{\ms Y}} 
\nc{\gsg}{\flag{X}} 
\nc{\gy}{\flag{Y}} 

\nc{\shasgm}[1]{\flag{\Sha^{#1}_{\sgm}}} 
\nc{\shasgf}[1]{\flag{\Sha^{#1}_{\sgf}}} 
\nc{\shasgfz}[1]{\flag{\Sha^{#1}_{\sgf,0}}} 
\nc{\shamy}[1]{\flag{\Sha^{#1}_{\my}}} 
\nc{\shafy}[1]{\flag{\Sha^{#1}_{\fy}}} 
\nc{\shagsg}[1]{\flag{\Sha^{#1}_{\Omega_{\gsg}}}} 

\nc{\valsgm}{\flag{\Omega_{\sgm}}} 
\nc{\valsgf}{\flag{\Omega_{\sgf}}} 

\nc{\valmy}{\flag{\Omega_{\my}}} 
\nc{\valfy}{{\flag{\Omega_{\fy}}}} 

\nc{\valgsg}{\flag{\Omega_{\gsg}}} 

\newtheorem{thm}{Theorem}[section]
\newtheorem{lemma}[thm]{Lemma}
\newtheorem{cor}[thm]{Corollary}
\newtheorem{prop}[thm]{Proposition}

\theoremstyle{definition}
\newtheorem{rem}[thm]{Remark}
\newtheorem{notation}[thm]{Notation}
\newtheorem{example}[thm]{Example}

\title{Local-global principles for curves over semi-global fields\\ }

\author[David Harbater]{David Harbater}
\address{David Harbater, Department of Mathematics,
 University of Pennsylvania\\
  Philadelphia, PA 19104-6395, USA}
\email{harbater@math.upenn.edu}

\author[Daniel Krashen]{Daniel Krashen}
\address{Daniel Krashen, Department of Mathematics, Rutgers University\\
Piscataway, NJ 08854, USA}
\email{daniel.krashen@rutgers.edu}

\author[Alena Pirutka]{Alena Pirutka}
\address{Alena Pirutka, Courant Institute of Mathematical Sciences,
New York University\\
New York, New York 10012, USA\newline
National Research University Higher School of Economics, Russian Federation}
\email{pirutka@cims.nyu.edu}

\date{July 13, 2020}
\thanks{
\!\!\!\textit{Mathematics Subject Classification} (2010): 11E72, 14H25, 20G15 (primary);  11G20, 12G05, 14G27 (secondary).}

\begin{document}
\maketitle
\vspace{-0.5cm}

\begin{abstract}
We investigate local-global principles for Galois cohomology, in the context of function fields of curves over semi-global fields.  This extends work of Kato's on the case of function fields of curves over global fields.
\end{abstract}

\section{Introduction}

A basic theme in field arithmetic is the passage from a field $\sgf$ to the
field of rational functions in one variable $\sgf(x)$. In this paper, we discuss how
local-global principles behave with respect to such transitions. In particular,
suppose we are given a collection of overfields $\sgf_v, v \in \Omega$, of
$\sgf$. A natural question to ask is, when one has knowledge of local-global
principles for $\sgf$ with respect to the overfields $\sgf_v$, what
local-global principles may be obtained for the field $\sgf(x)$ with respect to
the overfields $\sgf_v(x)$. More generally, one may also replace $\sgf(x)$ by
the function field $\sgf(\gy)$ of an algebraic curve $\gy/\sgf$.

Such questions have been asked before, primarily in the case that $\sgf$ is a global field, and some results have been obtained for the third Galois cohomology group \cite{Kato}, for Witt groups of quadratic forms \cite{ParSuj}, for the Brauer group \cite{LPS}, and for torsors over simply connected, semisimple groups of classical type \cite{ParPre}. A theme that has emerged in these investigations is that as one adjoins a transcendental element to a field, one expects to find local-global principles one degree higher than previously.  

The first main instance of this theme concerned the theorem of Albert-Brauer-Hasse-Noether, which states a local-global principle for the Brauer group of a global field $\sgf$, or equivalently for $H^2(\sgf,\mu_m)$.  Such a principle need not hold for $\sgf(\gy)$, if $\gy$ is a curve over a global field $\sgf$.  But in \cite{Kato}, a local-global principle was shown in this context for $H^3(\sgf(\gy),\mu_m^{\otimes 2})$, with respect to the places of $\sgf$.  

In recent years, local-global principles have been studied for semi-global fields; i.e., one-variable function fields $\sgf$ over a complete discretely valued field $\dvf$.  Such principles were shown in \cite{CPS} and \cite{HHK:H1} for $H^1(\sgf,G)$ in the context of a rational linear algebraic group $G$;
and in \cite{HHK:LGGC} for $H^n(\sgf,\mu_m^{\otimes{n-1}})$ with respect to the discrete valuations on $\sgf$ if $n \ge 2$.  The above theme suggests that one should expect a local-global principle for $H^n(\sgf(\gy),\mu_m^{\otimes{n-1}})$ with respect to the discrete valuations on a semi-global field $\sgf$ for sufficiently large values of $n \le \cd(\sgf(\gy))$.  We prove such a result in Theorem~\ref{LGP for line} for $n \ge 3$ in the case that $\gy$ is the projective line over $\sgf$, and for more general curves $\gy/\sgf$ in Theorem~\ref{LGP over local field} for $n \ge 3$ if $\dvf$ is a local field.

In several papers (see \cite{CH}, \cite{HSS},\cite{HS16}), local-global principles for semi-global fields $\sgf$ were considered with respect to a smaller class of discrete valuations; viz., those that are trivial on the underlying complete discretely valued field $\dvf$.  In that situation, a local-global principle holds less often, since the hypothesis is weaker.  In Theorem~\ref{gen fiber Sha} we consider local-global principles for $H^n(\sgf(x),\mu_m^{\otimes{n-1}})$ with respect to this smaller set of valuations on a semi-global field $\sgf$, showing that it fails for $n=4$ but holds under certain hypotheses (e.g., good reduction) for $n=3$.

This manuscript is organized as follows.  
We introduce the set-up and discuss some classical examples in Section \ref{setup}, followed by preliminary results in Section~\ref{lemmas}.  In Section~\ref{sgf section}, after establishing our notation, we discuss the relationship between two types of local-global principles on a function field $\fy = \sgf(\gy)$ over a semi-global field $\sgf$: with respect to the discrete valuations on $\fy$, and with respect to those on $\sgf$ (following Kato).  Our main results appear in Section~\ref{results}, first with respect to all the divisorial discrete valuations on $\sgf$ (in Section~\ref{all vals}) and then with respect to those that are trivial on the underlying complete discretely valued field $\dvf$ (in Section~\ref{gen fiber vals}). 

\medskip

\paragraph{\bf Acknowledgements}
The first author was partially supported by NSF collaborative FRG grant DMS-146373 and NSF grant DMS-1805439.  The second author was partially supported by NSF RTG grant DMS-1344994, 
NSF collaborative FRG grant DMS-1463901, and NSF grant DMS-1902144.
The third author was partially
supported by NSF grant DMS-1601680, by ANR grant ANR-15-CE40-0002-01, and by the Laboratory of Mirror
Symmetry NRU HSE, RF Government grant, ag.\ no.\ 14.641.31.0001.  

We thank R.~Parimala and V.~Suresh for providing us with the proof of  Theorem~\ref{sha mixed thm}, which strengthens a result in \cite{HHK:LGGC}, and is used in what follows.  
We thank 
J.-L.~Colliot-Th\'el\`ene for several discussions, and for comments on the manuscript that allowed us to generalize some of our results. 
We thank D.~Harari, J.~Hartmann, and K.~Smith for discussions about material in this paper.

\section{Definitions and Basic Examples} \label{setup}

Let  $\sgf$ be a field and let $\sgf_v, v\in \Omega$ be a collection of
overfields $\sgf_v\supseteq \sgf$. If $\fy/\sgf$ is a finite extension we set
$\fy_v=\fy\otimes_\sgf \sgf_v$ (which is not necessarily a field).  On the other hand, if $\gy/\sgf$ is a
geometrically integral algebraic variety and $\fy=\sgf(\gy)$, we set $\fy_v$  the
function field of the variety $\gy_{\sgf_v}$. For $\mu$ a $\Gal(\bar \fy/\fy)$-module
we set
\begin{equation}\label{Shagen}
\Sha^n_{\Omega}(\fy,\mu)=\ker\left[H^n(\fy, \mu)\to \prod_{v\in \Omega} H^n(\fy_v,
\mu)\right].
\end{equation}

\newcommand{\Jac}{\operatorname{Jac}}
We are interested  to determine this group, and, in particular, to determine whether or not this group is trivial for certain fields $\sgf$ and collections $\sgf_v, v\in \Omega$ of overfields. Our guiding principle is to understand how this group should depend on the quantities $n$, $i=\cd\,\sgf$, and $d=\dim\, \gy$.  We set $d=0$ if $\fy$ is a finite extension of $\sgf$.
We will eventually be mostly interested in the case where $\sgf$ is a function field in one variable over a complete discretely valued field, with a set of overfields described in Section~\ref{ccl} below, and $d=1$; that is, when $\fy = \sgf(\gy)$ for some
curve $\gy/\sgf$.

To begin investigating properties of the groups $\Sha^n_{\Omega}(\fy,\mu)$, we first consider some examples:
\subsection{Number fields} \label{snt}
Let $\sgf$ be a number field, and let
$\Omega$ be the set of
places on $\sgf$; the completion of $\sgf$ at $v \in \Omega$ will be denoted by $\sgf_v$.
We recall a few known cases of local-global principles for particular values of $d, n, \mu$. Let $\ell$ be a prime.

\begin{example} {$d = 0, n = 2, \mu = \mu_\ell$:}\\
As $d = 0$, it follows that $\fy$ is a finite
		extension of $\sgf$. We may identify $H^2(\fy, \mu_\ell)$ with $\Br(\fy)[\ell]$.
		In this case, $\Sha^2_{\Omega}(\fy,\mu_{\ell})=0$ 
		by the Albert-Brauer-Hasse-Noether theorem.
\end{example}

\begin{example} {$d = 1, n = 3, \mu = \mu_\ell^{\otimes 2}$:}\\
		Here, writing $\fy = \sgf(\gy)$ for a curve $\gy/\sgf$, Kato shows (in \cite[Theorem~0.8(2)]{Kato}) that $\Sha^3_{\Omega}(\fy,\mu_{\ell}^{\otimes 2}) = 0$.
\end{example}

\begin{example} {$d = 1, n = 2, \mu = \mu_\ell$:}\\
Write $\fy=\sgf(\gy)$ with $\gy$  a smooth projective curve over $\sgf$, and assume $\gy(\sgf) \neq \emptyset$.  In this case, $\Sha^2_{\Omega}(\fy, \mu_\ell)$ need not vanish.  To see this, first note that
$\Sha^2_{\Omega}(\fy, \mu_\ell) \subseteq H^2(\fy, \mu_\ell) = \Br(\fy)[\ell]$ and consider the following  commutative diagram
		\[\xymatrix{
	&		\Br(\fy)[\ell] \ar[r] \ar[d] & \coprod_{y \in \gy^{(1)}} H^1(\kappa(y), \mu_\ell) \ar[d] \\
	&		\prod_{v \in \Omega} \Br(\fy_v)[\ell] \ar[r] & \prod_v \coprod_{z\in \gy_{\sgf_v}^{(1)}} H^1(\kappa(z), \mu_\ell).
&		}\]
If $L = \kappa(y)$ is the residue field of a closed point $y$ of $\gy$, then 
$L/\sgf$ is a finite extension and any nontrivial element of $H^1(L, \mu_\ell)$  must be nontrivial locally at infinitely many places $v \in \Omega$ by the Grunwald-Wang Theorem (or by application of Chebotarev density). We deduce that any class $\alpha\in  \Sha^2_{\Omega}(\fy, \mu_\ell)$ lies in the subgroup $\Br(\gy)[\ell] \subseteq \Br(\fy)[\ell]$ of unramified elements. It follows that
		$$\Sha^2_{\Omega}(\fy, \mu_\ell) = \ker ( \Br(\gy)[\ell] \to \prod_v \Br(\gy_v)[\ell]). $$
	Since $\gy(\sgf)\neq \emptyset$, we have  $\Br(\gy) = \Br(\sgf) \oplus H^1(\sgf, \Jac(\gy))$ from the Hochschild-Serre spectral sequence.
Since $\Sha^2_{\Omega}(\sgf, \mu_\ell) = 0$, it follows that $\Sha^2_{\Omega}(\fy, \mu_\ell) = \Sha^1_{\Omega}(\sgf, \Jac(\gy))[\ell]$ is the $\ell$-torsion of the (classical) Tate-Shafarevich group of the Jacobian of $\gy$.  In particular, it can be non-trivial; e.g., if $\gy$ is the Jacobian of the Selmer curve $3x^3+4y^3+5z^3$ over $\sgf=\Q$, and $\ell=3$.
\end{example}

\subsection{Curves over complete fields}\label{ccl}
We now recall some local-global principles in the context of semi-global fields.  

Let $\dvf$ be the fraction field of a discrete valuation ring $\dvr$, and let 
$\sgf$ be the function field of an absolutely irreducible $\dvf$-variety.
Equivalently, $\sgf$ is a finitely generated field extension of $\dvf$ in which $\dvf$ is algebraically closed.  A normal (resp.\ regular) {\em model} of $\sgf$ is an integral $\dvr$-scheme $\sgm$ with function field~$\sgf$ that is flat and projective over~$\dvr$ of relative dimension $\td_\dvf(\sgf)$, and that is normal (resp.\ regular) as a scheme.  A normal model always exists; e.g., the normalization of a projective closure of an affine $\dvr$-variety with function field $\sgf$.  

In particular, consider the case of transcendence degree one; i.e., where $\sgm$ is a relative $\dvr$-curve.  A regular model of $\sgf$ then exists if $\dvr$ is excellent, by the main theorem of \cite{Lip78}.  If $\dvr$ is complete then $\dvr$ is excellent, and so there is a regular model of $\sgf$.  
In this situation, where $\dvr$ is complete and $\td_\dvf(\sgf)=1$,   
we call $\sgf$ a {\em semi-global} field over $\dvf$.

Let $m$ be a positive integer that is not divisible by the
characteristic of $\dvf$.
We consider the following two situations:

\subsubsection{Local-global principle with respect to scheme theoretic points of the closed fiber} \label{lgp pts}
Let $\Omega$ be the set of scheme-theoretic points on the closed fiber of $\sgm$.
For $x \in \Omega$, let $\sgf_x =
\frf(\wh {\ms O}_{\sgm, x})$ be the fraction field of the
complete local ring at $x$; this is an overfield of $K$.
If $n \ge 2$, then the associated kernel $\Sha_{\Omega}^n(\sgf, \mu_m^{\otimes n-1})$
is trivial by \cite[Theorem~3.2.3(i)]{HHK:LGGC}.  On the other hand if $n=1$, it follows from \cite[Corollary~6.5]{HHK:H1} that $\Sha_{\Omega}^1(\sgf,\Z/m\Z)$ is finite but not necessarily trivial.

\subsubsection{Local-global principle with respect to divisorial discrete valuations}
\label{lgp dvr}

If $\dvf$ is of equal characteristic, and $\Omega = \sgm^{(1)}$  is the set of codimension $1$ points of $\sgm$, then
$\Sha_{\Omega}^n(\sgf, \mu_m^{\otimes n-1}) = 0$,
for $n\geq 2$, by \cite[Theorem~3.3.6]{HHK:LGGC}.  This result has now been shown to carry over to assertions in mixed characteristic as well.  In particular, see Theorem~\ref{sha mixed thm} below, which was shown by R.~Parimala and V.~Suresh, and which we use in the sequel.  Other extensions have been announced as well (see \cite{Sakagaito}). However, the group $\Sha_{\Omega}^2(\sgf, \mu)$ is not necessarily trivial for $\mu$ a finite Galois module (different from $ \mu_m^{\otimes n-1}$) and $\sgf$ the function field of a curve over a local field (see \cite[Corollaire~5.3(iii)]{CPS16}). Concerning the case of $n=1$, 
$\Sha_{\Omega}^1(\sgf,\Z/m\Z)$ is not necessarily trivial (see \cite[Section 6]{CPS}); moreover, it is equal to the corresponding obstruction group in \ref{lgp pts} above, by \cite[Proposition~5.1, Corollary~5.3]{HHK:H1}.

\section{Function fields of rational curves} \label{lemmas}

Let $\sgf$ be a field, and let $\mu$ be a finite commutative \'etale group scheme over $\sgf$, whose order is prime to the characteristic of $\sgf$.  
In Proposition~\ref{ShaD} below, we relate $\Sha^n_{\Omega}(\sgf,\mu)$ to $\Sha^n_{\Omega}(\fy,\mu)$, where $\fy=\sgf(t)$ and where $\Omega$ is a collection of overfields of $\sgf$.
First we prove a general lemma.

\begin{lemma} \label{inseparably ramified}
Let $L/\sgf$ be an arbitrary field extension, and let $\fy$ be the function field of
a smooth, geometrically integral curve $\gy/\sgf$. Let $v$ be a valuation on $\fy$ corresponding to a closed point $P$ on $\gy$, and let $Q$ be a closed point of $\gy_L := \gy \times_\sgf L$ above $P$. Then we have an inclusion of discrete valuation rings $\ms O_{\gy, P} \subseteq \ms O_{\gy_L, Q}$, and the ramification index (i.e., valuation in $\ms O_{\gy_L, Q}$ of a uniformizer of $\ms O_{\gy, P}$) is equal either to $1$, or if $\sgf$ has finite characteristic $p$, to a power of $p$.
\end{lemma}

\begin{proof}
First consider the case in which $\gy = \mathbb P^1_\sgf$, so that $\fy=\sgf(t)$.  The assertion is trivial if $P$ is the point at infinity, so assume that $P$ is a closed point on the affine line.  Thus $P$ is given by a monic irreducible polynomial
$f(t) \in \sgf[t]$.  Let the irreducible factorization of $f(t)$ in $L[t]$ be
$\prod_j f_j(t)^{r_j}$ for some distinct monic irreducible polynomials $f_j(t) \in L[t]$ and some non-negative integers $r_j$.  The integers $r_j$ are the ramification indices of the points $Q_j \in \gy_L$ that lie over $P \in \gy$.  We want to show that each $r_j$ either equals $1$ or is a power of $p$ if $\sgf$ has characteristic $p>0$.

Let $\bar L$ be an algebraic closure of $L$, and let $\bar\sgf$ be the algebraic closure of $\sgf$ in $\bar L$.
Since $f$ is irreducible over $\sgf$, it factors over $\bar\sgf[t]$ as
$\prod_i (t-a_i)^{s_i}$ for some distinct elements $a_i \in \bar \sgf$ and some non-negative integers $s_i$ that are each equal to $1$ or a power of $p = \cha(\sgf)$.  So this is also the irreducible factorization of $f(t)$ in $\bar L[t]$.

For each $j$, the roots of $f_j(t)$ in $\bar L$ are a subset of the roots of $f$ in $\bar L$, viz. the elements $a_i$.  Choose one of these roots $a_i$ of $f_j(t)$, say with multiplicity $m_i$ in $f_j$.  So $f_j(t)$ is the minimal polynomial of $a_i$ over $L$, and $a_i$ is not the root of any other $f_{j'}$. The multiplicity of $a_i$ as a root of $f(t) = \prod_j f_j(t)^{r_j}$ in $\bar L$ is thus $m_ir_j$.  Thus $m_ir_j = s_i$, and so $r_j$ is equal either to $1$ or to a power of $p$.  This completes the proof in the case that $\gy = \mathbb P^1_\sgf$.

For the general case, since $\gy$ is smooth over $\sgf$, there is a uniformizer on $\gy$ at $P$.  This defines a morphism $\pi:\gy \to \mathbb P^1_\sgf$ that is \'etale at $P$, so that the ramification index of $P$ over the point $P' = (t=0)$ on $\mathbb P^1_\sgf$ is equal to $1$.  We then have a pullback diagram
		\[\xymatrix{
			\gy_L \ar[r]^{\pi'} \ar[d]^{\sigma'} & \mathbb P^1_L \ar[d]^{\sigma}    \\
           \gy \ar[r]^\pi & \mathbb P^1_\sgf \ .
		}\]
For any point $Q \in \gy_L$ over $P \in \gy$, its image $Q'$ in $\mathbb P^1_L$ lies over $P'$.  So by the above special case, the ramification index of $\sigma$ at $Q'$ is either $1$ or a power of $p$.  And since $\pi$ is \'etale at $Q$, its pullback $\pi'$ is \'etale at $Q'$; and so $\pi'$ has ramification index $1$ there.  Thus the ramification index of $\pi\sigma' = \sigma\pi'$ at $Q$ is also of this form.  Hence so is the ramification index of $\sigma'$ at $Q$, as asserted.
\end{proof}

Turning to our situation, recall that there is
an exact sequence
$$0 \to H^n(\sgf, \mu)\to H^n(\sgf(\mathbb P^1), \mu) \to \coprod_{y \in
(\mathbb P^1_\sgf)^{(1)}} H^{n-1}(\kappa (y), \mu(-1)) \to H^{n-1}(\sgf,
\mu(-1)) \to 0,$$
where the second map is the sum of the residue maps, and where the map on the right is the sum of the corestriction
maps (e.g., see \cite[Part 1, Theorem~9.2]{GMS:CohInv}; or
\cite[Theorem~6.9.1]{GS:CSAGC} in case $\mu = \mu_m^{\otimes r}$). We will refer to this as the {\em Faddeev sequence}.

Choosing the point at infinity defines a splitting of this
rightmost map, and one obtains a short exact sequence: \[0 \to H^n(\sgf, \mu)
\to H^n(\sgf(\mathbb P^1), \mu) \to \coprod_{y \in \mathbb A^{1\,(1)}_\sgf}
H^{n-1}(\kappa(y),
\mu(-1)) \to 0.\]
We then have the following proposition for a purely transcendental extension $\fy$ of a field $\sgf$.

\begin{prop}\label{ShaD}{Let $\sgf$ be a field; let $\sgf_v, v\in \Omega,$ be a
collection of overfields of $\sgf$; and let $\fy=\sgf(t)$ be a purely transcendental extension of $\sgf$.  Let $\mu$ be a finite commutative \'etale group scheme over $\sgf$, whose order is prime to the characteristic of $\sgf$. Take  $n\geq 1$ an integer.  We
have the following exact sequence:
$$0\to \Sha^n_{\Omega}(\sgf,\mu)\to \Sha^n_{\Omega}(\fy,\mu)\to \coprod_{y\in
\mathbb A^{1\,(1)}_\sgf} \Sha^{n-1}_{\Omega} (\kappa(y),\mu(-1)) \to 0.$$
}
\end{prop}

\begin{proof}
Applying the splitting of the Faddeev sequence to the fields $\sgf$ and $\sgf_v$, $v\in \Omega$,  we obtain a commutative diagram:
\begin{equation}
\label{dFaddeev}
\xymatrix@C=0.7cm{
	0\ar[r]& H^n(\sgf, \mu)\ar[d]\ar[r] & H^n(\fy, \mu)\ar[r]^(0.3){\phi}\ar[d]&\coprod_{y\in
	(\mathbb A^1_\sgf)^{(1)}}  H^{n-1}(\kappa(y), \mu(-1))\ar[d] \ar[r] & 0\\
	0\ar[r]& \prod_{v} H^n(\sgf_v, \mu)\ar[r]& \prod_v H^n(\fy_v, \mu)\ar[r]&
	\prod_v\coprod_{z\in (\mathbb A^1_{\sgf_v})^{(1)}} H^{n-1}(\kappa(z), \mu(-1)) \ar[r] & 0,}
\end{equation}
where $\fy_v=\sgf_v(t)$, and the rightmost vertical map is defined as follows:  For a closed point $y \in \mathbb A^1_\sgf$, and $v \in \Omega$, we
may consider the finite set of closed points of $\mathbb A^1_{\sgf_v}$ lying
over $y$. These correspond to the maximal ideals of the ring $\kappa(y)
\otimes_{\sgf} \sgf_v$. 
For each such point $z \in \mathbb A^1_{\sgf_v}$, we
have an extension of local rings $\ms O_{\mbb A^1_{\sgf_v}, z}/\ms O_{\mbb A^1_{\sgf}, y}$. 
We set $e_z$ to be the ramification index of this extension. We then
define the rightmost vertical map to be the sum of the maps $H^{n-1}(\kappa(y),
\mu(-1)) \to H^{n-1}(\kappa(z), \mu(-1))$ given by the  natural pullback and
multiplication by $e_z$. The fact that the diagram~(\ref{dFaddeev}) commutes
follows from \cite[Proposition~8.2]{GMS:CohInv}. By Lemma~\ref{inseparably ramified}, each ramification degree $e_z$ is a power of the characteristic, and since this is prime to the order of $\mu$, it follows that the kernel of this map is the same as the kernel of the natural pullback maps.

{From} this, we deduce a short exact sequence of kernels
$$0\to   \Sha^n_{\Omega}(\sgf,\mu)\to \Sha^n_{\Omega}(\fy,\mu)\to \coprod_{y\in
\mathbb A^{1\,(1)}_\sgf} \Sha^{n-1}_{\Omega} (\kappa(y),\mu(-1)).$$
It remains to show that the rightmost map  is surjective.

Choose $\beta\in \coprod \Sha^{n-1}_{\Omega} (\kappa(y),\mu(-1))$ and $\alpha\in H^n(\fy, \mu)$ with $\phi(\alpha)=\beta$. Choose $p\in \mathbb A^1(\sgf)$ a rational point disjoint from the support of $\beta$. In particular, $\alpha$ is unramified at $p$. Put $\alpha'=\alpha-\alpha(p)_\fy$, where $\alpha(p)$ is the specialization of $\alpha$ at $p$ and $\alpha(p)_\fy$ is the pullback of the class $\alpha(p)$ via the inclusion $\sgf\subset \fy$.  Note that  $\phi(\alpha')=\beta$.
We claim that $\alpha'\in \Sha^n_{\Omega}(\fy,\mu)$. Let $\alpha'_{\fy_v}$ be the image of $\alpha'$ in $H^n(\fy_v, \mu)$. By the commutativity of the diagram (\ref{dFaddeev}), $\alpha'_{\fy_v}$ maps to zero in  $\coprod_{z\in (\mathbb A^1_{\sgf_v})^{(1)}} H^{n-1}(\kappa(z), \mu(-1))$, hence $\alpha'$ comes from a class in  $H^n(\sgf_v, \mu)$. In particular, it is determined by the value of its specialization at any $\sgf_v$-rational point. But its specialization at $p$ is $\alpha'(p)=\alpha(p)-\alpha(p)=0$, so that $\alpha'_{\fy_v}=0$ and $\alpha'\in \Sha^n_{\Omega}(\fy,\mu)$.
\end{proof}

\section{Curves over semi-global fields} \label{sgf section}

Let $\sgf$ be a semi-global field over a complete discretely valued field $\dvf$; i.e., the function field of a projective $\dvf$-curve (as always, assumed absolutely irreducible).  Let $\fy$ be the function field of a projective curve $\gy$ over $\sgf$.  In this situation, we can consider local-global principles for $\fy$ with respect to the valuations on $\sgf$, or with respect to those on $\fy$.  We relate the corresponding $\Sha$ groups in Propositions~\ref{comparisonHnr} and~\ref{Sha agree} below. 

In this section we will be interested in the following sets of valuations.  Let $\valsgf$ be the set of divisorial discrete valuations on $\sgf$; i.e.\ valuations corresponding to codimension one
points of some normal model of $\sgf$ over the valuation ring $\dvr$ of $\dvf$.  In fact, a regular model $
\sgm$ of $\sgf$ exists, as noted in Section~\ref{ccl}; and we write 
$\valsgm \subseteq \valsgf$ for the set of valuations on $\sgf$ corresponding to codimension one points on $\sgm$.  
For $v \in \valsgf$, let $\sgf_v$ be the completion of $\sgf$ at $v$ and let  $\fy_v=\fy\otimes_\sgf \sgf_v$.
Similarly, we let $\valfy$ be the set of divisorial discrete valuations on $\fy$, so that any $v\in \valfy$ corresponds to a codimension one point on some normal  model  of $\fy$ over $\dvr$. If $\my$ is a fixed such normal model of $\fy$, we write $\valmy \subseteq \valfy$ for the set of valuations on $\fy$ corresponding to codimension one points on $\my$.

\begin{notation} \label{notn}
For the rest of this paper we fix an integer $m>1$ prime to the residue characteristic of $\dvf$, and we let $\zeta_m$ be a primitive $m$-th root of unity in an algebraic closure of $\dvf$.  For any field $L$ containing $\dvf$ and any integer $n \ge 0$, we write $H^n(L) := H^n(L,\mu_m^{\otimes (n-1)})$.
With $\fy$ as above, we write
$$ H^n_{\nr}(\fy)=\bigcap_{v\in \valfy} \ker[\partial^n:  H^n(\fy)\to H^{n-1}(\kappa(v))]$$
for the $n$-th unramified cohomology of $\fy$ with coefficients in $\mu_m^{\otimes (n-1)}$.  For 
a normal proper model $\my$ of $\fy$ over $\dvr$ we define
$$ H^n_{\nr}(\fy/\my)=\bigcap_{v\in \valmy} \ker[\partial^n:  H^n(\fy)\to H^{n-1}(\kappa(v))],$$
which by definition contains $H^n_{\nr}(\fy)$.  
We also consider the following $\Sha$ groups:
$$\shafy n(\fy)= \Sha^n_{\valfy}(\fy,\mu_m^{\otimes (n-1)})
=\ker[H^n(\fy)\to \prod_{v\in\valfy}  H^n(\fy_v)],$$
$$\shasgf n(\fy) = \Sha^n_{\valsgf}(\fy,\mu_m^{\otimes (n-1)})
= \ker[H^n(\fy)\to \prod_{v\in\valsgf}  H^n(\fy_v)],$$
$$\shasgf n(\sgf) = \Sha^n_{\valsgf}(\sgf,\mu_m^{\otimes (n-1)})
= \ker[H^n(\sgf)\to \prod_{v\in\valsgf}  H^n(\sgf_v)].$$
\end{notation}

We have that $\shafy n(\fy)\subseteq H^n_{\nr}(\fy)$, since the residue map factors through the completion.

In this and the following section, we will want to use that $\shasgf n(\sgf)$ is trivial for $n>1$.  In  
Theorem~3.3.6 of \cite{HHK:LGGC}, this triviality was shown to hold if the underlying complete discretely valued field $\dvf$ is equicharacteristic; and moreover that $\Sha^n_{\valsgf}(\sgf,\mu_m^{\otimes r})$ vanishes for all $r$ if $[\sgf(\zeta_m) : \sgf]$ is prime to $m$.  
(In fact, it was shown there that for any regular model $\sgm$ of $\sgf$,
the corresponding {\em a priori} larger obstruction groups with respect to $\valsgm$ vanish.)  Several years ago, R.~Parimala and V.~Suresh proved that if $[\sgf(\zeta_m) : \sgf]$ is prime to $m$ then $\Sha^n_{\valsgf}(\sgf,\mu_m^{\otimes r})$ vanishes for all $r$, even if $\dvf$ is not assumed to be equicharacteristic.  Their argument has been unpublished until now (although it has been cited in the literature; e.g.\ \cite[p.~4372]{Hu}); and we thank them for allowing us to provide it here, in Lemma~\ref{la:coho cl van} and Theorem~\ref{sha mixed thm} below.

\begin{lemma}[Parimala-Suresh] \label{la:coho cl van}
Let $A$ be a complete two dimensional local ring with maximal ideal $\frak m = (\pi,\delta)$, residue field $\kappa$ and field of fractions $L$. Let $m$ be a positive integer that is invertible in $A$.  Let $\xi = \sum (a_{i1}) \cdot (a_{i2}) \cdots (a_{in}) \in H^n(L, \mu_m^{\otimes n})$ with $a_{ij} \in A$ representing $(a_{ij}) \in L^\times/(L^\times)^m = H^1(L,\mu_m)$. Suppose that $\pi$ and $\delta$ are the only primes that divide $a_{ij}$ for all $i,j$. If the image of $\xi$ is zero over $L_\pi$ and $L_\delta$, then $\xi$ is zero.
\end{lemma}

\begin{proof}
Since each $a_{ij}$ is of the form $u\pi^\epsilon\delta^{\epsilon'}$ for some unit $u$ in $A$, $0 \le \epsilon,\epsilon' < m$ 
and $(ab) \cdot (c) = (a) \cdot (c) + (b) \cdot (c)$ and $(a) \cdot (-a) = 0$, it follows that
$$\xi = \xi_0+\xi_1 \cdot(\pi)+\xi_2 \cdot(\delta)+\xi_3 \cdot(\pi)\cdot(\delta),$$
where $\xi_0$ is a sum of symbols of the form $(t_1)\cdot(t_2)\cdots(t_{n})$; $\xi_1,\xi_2$ are sums of symbols of the form  $(u_1)\cdot(u_2)\cdots(u_{n-1})$; and $\xi_3$ is a sum of symbols
$(w_1)\cdots(w_{n-2})$, with $t_i,u_j,w_s$ units in $A$.
Since $m$ is a unit in $A$, we have $\xi_0 \in
H_{\rm et}^{n}(A,\mu_m^{\otimes n})$; $\xi_1,\xi_2 \in
H_{\rm et}^{n-1}(A,\mu_m^{\otimes n-1})$; and $\xi_3 \in H_{\rm et}^{n-2}(A,\mu_m^{\otimes n-2})$. Let $\partial_\pi : H^n(L,\mu_m^{\otimes n}) \to H^{n-1}(\kappa(\pi),\mu_m^{\otimes n-1})$ be the
residue homomorphism.  Then $\partial_\pi(\xi) = \overline\xi_1 - \overline\xi_3 \!\cdot\! (\overline\delta)$, where bar denotes the image in $A/(\pi)$.  The image of $\xi$ in $H^n(L_\pi, \mu_m^{\otimes n})$ is zero, and so $\partial_\pi(\xi) = \overline\xi_1 - \overline\xi_3 \!\cdot\! (\overline\delta) = 0$. Since $A/(\pi)$ is a complete discrete valuation ring with field of fractions $\kappa(\pi)$ and residue field $\kappa$, we have the residue homomorphism $\partial_{\overline\delta} : H^{n-1}(\kappa(\pi), \mu_m^{\otimes n-1}) \to H^{n-2}(\kappa, \mu_m^{\otimes n-2})$. Here
$\overline\delta$ is a parameter in $A/(\pi)$; so $\partial_{\overline\delta}(\overline\xi_1) = 0$ and $\partial_{\overline\delta}(\overline\xi_3\!\cdot\!(\overline\delta)) = \overline{\overline\xi_3}$, where $\overline{\overline\xi_3}$ denotes
the image in $A/\frak m=\kappa$.
But $\overline\xi_1-\overline\xi_3\!\cdot\!(\overline\delta)=0$, so $\overline{\overline\xi_3} =\partial_{\overline\delta}(\overline\xi_1-\overline\xi_3\!\cdot\!(\overline\delta))=0$. Since $A$ is a
complete regular local ring, we have $\xi_3 = 0$. The element $\overline\xi_1-\overline\xi_3\!\cdot\!(\overline\delta)$ is zero in $H^{n-1}(\kappa(\pi),\mu_m^{\otimes n-1})$, and so
$\overline\xi_1$ is zero in $H^{n-1}(\kappa(\pi),\mu_m^{\otimes n-1})$. Since $A/(\pi)$ is a complete discrete valuation ring with
field of fractions $\kappa(\pi)$, the map $H_{\rm et}^{n-1}(A/(\pi),\mu_m^{\otimes n-1}) \to H^{n-1}(\kappa(\pi),\mu_m^{\otimes n-1})$ is injective and hence $\overline \xi_1$
is zero in $H_{\rm et}^{n-1}(A/(\pi),\mu_m^{\otimes n-1})$. Since $\pi$ is in the maximal ideal of the complete local ring $A$, it follows that $\xi_1$ is zero. Similarly, using that the image of $\xi$ is zero in $L_\delta$, it follows that $\xi_2$ is zero. Finally, since $\xi_0$ is zero over $L_{\pi}$, we deduce that $\overline{\xi_0}$ is zero in $H_{\rm et}^{n}(A/(\pi),\mu_m^{\otimes n})\hookrightarrow H^{n}(\kappa(\pi),\mu_m^{\otimes n})$,  so that, similarly to the above, $\xi_0=0$.
  Hence $\xi = 0$.
\end{proof}

\begin{thm}[Parimala-Suresh] \label{sha mixed thm}
Let $\dvf$ be a complete discretely valued field with residue field $\kappa$, and let $\sgf$ be the function field of a regular curve over $\dvf$.
Let $m$ be a positive integer that is not divisible by the characteristic of $\kappa$, and assume that  $[\sgf(\zeta_m) : \sgf]$ is prime to $m$.
Then $\shasgf n(\sgf,\mu_m^{\otimes r}), n \ge 2$, is trivial for every integer $r$.  Equivalently,
the natural homomorphism
\[\iota:H^n(\sgf,\mu_m^{\otimes r}) \to \prod_{v \in \valsgf}  H^n(\sgf_v,\mu_m^{\otimes r})\]
is injective for $n \ge 2$.
\end{thm}

\begin{proof}
First consider the case where $r=n$, and let $\xi \in H^n(\sgf, \mu_m^{\otimes n})$.
It follows from the Bloch-Kato conjecture \cite{Voe11} that 
\[\xi =  \sum_i (a_{i1}) \cdot \ldots \cdot (a_{in}),\] where $(a_{ij}) \in H^1(\sgf,\mu_m)$ for some $a_{ij} \in \sgf^\times$. 
Let $\sgm$ be a regular proper model of $\sgf$ over the valuation ring $R$ of $\dvf$ such that $\cup {\rm Supp}_{\sgm} (a_{ij})$ is a union of regular curves with normal crossings.  Such a regular model exists by \cite[page 193]{Lip75} and \cite{Lip78}.

Suppose that the image of $\xi$ in each $H^n(\sgf_v,\mu_m^{\otimes n})$ is zero. Let $P$ be a closed point of $\sgm$ (thus lying on the closed fiber), and let $A=\hat{\ms O}_{\sgm,P}$ be the completion of the local ring of $\sgm$ at $P$. 
By the normal crossings hypothesis, we may apply
Lemma~\ref{la:coho cl van} to $A$ for suitable elements $\pi,\delta$; and so the image of $\xi$ in $H^n(\sgf_P, \mu_m^{\otimes n})$ is zero. 
Since $[\sgf(\zeta_m) : \sgf]$ is prime to $m$, it follows from
\cite[Theorem~3.2.3(i)]{HHK:LGGC} that $\xi = 0$, proving this case.

For the case of a general value of $r$, let $\xi \in H^n(\sgf,\mu_m^{\otimes r})$ be in the kernel of $\iota$. Since $H^n(\sgf(\zeta_m),\mu_m^{\otimes r}) \simeq H^n(\sgf(\zeta_m),\mu_m) \simeq H^n(\sgf(\zeta_m),\mu_m^{\otimes n})$, the image of $\xi$ in $H^n(\sgf(\zeta_m),\mu_m^{\otimes r})$ is zero, by the above special case.  Since $[\sgf(\zeta_m) : \sgf]$ is prime to $m$, it follows
that $\xi = 0$.
\end{proof}

Of course the hypothesis that $[\sgf(\zeta_m) : \sgf]$ is prime to $m$ is automatic if $m$ is prime. The argument below shows that we can always make the latter assumption in order to establish that $\shasgf n(\sgf,\mu_m^{\otimes n-1} )$ is trivial. We thank Jean-Louis Colliot-Th\'el\`ene for pointing that out.

\begin{cor}\label{mtol}
Let $\dvf$ be a complete discretely valued field with residue field $\kappa$, and let $\sgf$ be the function field of a regular curve over $\dvf$.
Let $m$ be a positive integer that is not divisible by the characteristic of $\kappa$.
Then $\shasgf n(\sgf,\mu_m^{\otimes n-1}), n \ge 2$, is trivial.
\end{cor}

\begin{proof}
By Theorem \ref{sha mixed thm} above, the assumption holds if $m$ is prime.
In the general case, by decomposing $m$ into a product of prime powers, we may assume  that $m=\ell^r$ is a power of a prime number.
Using the Bloch-Kato conjecture \cite{Voe11}, one has that the map  
$$H^{n}(\sgf ,\mu_{\ell^r}^{\otimes n-1}))\stackrel{\iota_r}{\to} H^n(\sgf, \mathbb Q_{\ell}/\mathbb Z_{\ell}(n-1))$$ 
is injective.
(See the argument that was given in \cite[18.4(c)]{MS} in the case $n=3$.  Namely, 
since $\ker\iota_r \simeq H^{n-1}(\sgf, \mathbb Q_{\ell}/\mathbb Z_{\ell}(n-1))/\ell^r$
by the Kummer exact sequence, it suffices to show that $H^{n-1}(\sgf, \mathbb Q_{\ell}/\mathbb Z_{\ell}(n-1))$ is $\ell$-divisible.  By Bloch-Kato this is isomorphic to $\varinjlim_i K_{n-1}^M \sgf/{\ell^i}$, where the transition maps are multiplication by $\ell$; and this is $\ell$-divisible.)
Hence the group
 $H^{n}(\sgf,\mu_{\ell^r} ^{\otimes n-1})$ is the $\ell^r$-torsion subgroup of the group  $H^n(\sgf, \mathbb Q_{\ell}/\mathbb Z_{\ell}(n-1))$.  
By Theorem \ref{sha mixed thm} above (applied to the group $\shasgf n(\sgf,\mu_{\ell}^{\otimes n-1})$), the map
 $$H^{n}(\sgf,\mu_{\ell}^{\otimes n-1})) \to \prod_{v\in\valsgf}\ H^{n}(\sgf_v,\mu_{\ell}^{\otimes n-1}))$$ is injective; hence the kernel of the map 
$$H^{n}(\sgf,\mathbb Q_{\ell}/\mathbb Z_{\ell}(n-1)) \to \prod_{v\in\valsgf}\ H^{n}(\sgf_v, \mathbb Q_{\ell}/\mathbb Z_{\ell}(n-1))$$ 
has no $\ell$-torsion. Hence it has no $\ell^r$-torsion. This gives that the map
$$H^{n}(\sgf,\mu_{\ell^r}^{\otimes n-1})) \to \prod_{v\in\valsgf}\ H^{n}(\sgf_v,\mu_{\ell^r}^{\otimes n-1}))$$
is injective as well, as desired.
\end{proof}

\begin{lemma} \label{flattening}
Consider a morphism of integral Noetherian schemes $f: V \to W$, and let $v$ be a discrete valuation on the function field $\kappa(V)$ corresponding to a prime divisor $D \subset V$. If the restriction $w$ of $v$ to the function field $\kappa(W)$ of $W$ is nontrivial, then there is a commutative diagram
\[\xymatrix{
  \til V \ar[r] \ar[d] & V \ar[d] \\
  \til W \ar[r] & W
}\]
whose horizontal maps are birational morphisms, together with a prime divisor $\til D$ on $\til V$ that induces the valuation $v$, and whose image in $\til W$ is a prime divisor $\til E$ giving rise to $w$.

That is, the restriction of a divisorial valuation, if nontrivial, is also divisorial.
\end{lemma} 

\begin{proof}
Let $d$ be the dimension of the generic fiber of $f$.
By \cite[\href{https://stacks.math.columbia.edu/tag/087E}{Lemma 087E}]{stacks-project}, we may find a blowup along some ideal sheaf $\til W \to W$ such that if we take the strict transform $\til V$ of $V$ with respect to the blowup, the morphism $\til V \to \til W$ is flat.
Note that we also have $\til V \to V$ is a blowup along the corresponding inverse ideal sheaf by \cite[\href{https://stacks.math.columbia.edu/tag/080E}{Lemma 080E}]{stacks-project}, and therefore is a birational morphism of integral Noetherian schemes. Since $\til V \to V$ is proper, by the valuative criterion, we may find a divisor $\til D$ of $\til V$ giving rise to $v$, whose image in $V$ is $D$. Now, consider the image $\til \Delta$ of $\til D$ in $\til W$.
Since the valuation $v$ is nontrivial on $\kappa(W)$, it follows that the morphism $D \to W$ is not dominant; hence neither is $\til D \to \til W$, and so $\til \Delta \neq \til W$. Since $\til V \to \til W$ is flat of relative dimension $d$, we have $\dim (\til \Delta \times_{\til W} \til V) = \dim \til \Delta + d \leq \dim W + d - 1$.
But since $\til D \subseteq \til \Delta \times_{\til W} \til V$,
and $\dim \til D = \dim V - 1 = \dim W + d - 1$, it follows that 
$\dim(\til \Delta \times_{\til W} \til V) = \dim W + d - 1$ and hence
$\dim \til \Delta = \dim W - 1$.
 In particular $\til \Delta$ is a divisor and since the center of $w$ must contain $\til \Delta$, it must be equal to $\til \Delta$ as desired.
\end{proof}

\begin{prop}\label{comparisonHnr}
Let $\fy$ be  the function field of a regular projective curve over a semi-global field $\sgf$.  Then
$$\shasgf n(\fy)\subseteq \shafy n(\fy) \subseteq H^n_{\nr}(\fy), \text{ for all }n\geq 3.$$
\end{prop}

\begin{proof}
The second containment was observed following Notation~\ref{notn}.  For the first containment,
let $\alpha\in \shasgf n(\fy)$.  Consider $w\in  \valfy$, corresponding to a codimension one point $y$ on some normal model $\my$ of $\fy$.  We wish to show that 
the image $\alpha_y$ of $\alpha$ in $H^n(\fy_w)$ vanishes. 
Let $D_y\subset \my$ be the closure of $y$ in $\my$. Let $\sgm$ be a regular model of $\sgf$.
 Note that the rational map $\my\dashrightarrow  \sgm$ extends to a regular map $\ms U\to \sgm$ on some open set $\ms U\subseteq \my$ containing $y$.
 
Assume first that $D_y$ is a vertical divisor; i.e., $y$ is in a fiber of $\pi$ over $x\in \sgm$. By Lemma~\ref{flattening}, after changing the model $\sgm$ we may assume that $x$ is a codimension $1$ point of $\sgm$, and $y$ is the generic point of one of the components of the closed fiber of  ${\ms U}\times_{\sgm}  {\ms O}_{\sgm, x}$. Let $v=v_x\in \valsgf$ be the corresponding divisorial valuation.
Then $y$ is also the generic point of one of the components of the closed fiber of   ${\ms U}\times_{\sgm} \wh {\ms O}_{\sgm, x}$, so that the
map $H^n(\fy)\to H^{n}(\fy_w)$ factors through $H^n(\fy_v)$.
 Since $\alpha$ maps to zero in $H^n(\fy_v)$, we obtain that the image of $\alpha$ in $H^{n}(\fy_w)$ is zero.

Assume next that $D_y$ is a horizontal divisor; i.e., the restriction to $\sgf$ of the discrete valuation corresponding to $y$ is trivial.  Thus we can
view $y$ as a codimension one point of the generic fiber $\gy$ of $\my$ over $\sgf$, 
whose residue field $L=\kappa (D_y)$ is a finite extension of $\sgf$. Let $v_z$ be a place on $L$ corresponding to a codimension $1$ point $z$ on some regular proper model $\tilde D_y\to \sgm$ of $L$, and let $x$ be the image of $z$ in $\sgm$. Again by Lemma~\ref{flattening}, after changing the model $\sgm$ along with $\tilde D_y$, we may assume  that $x$ is a codimension $1$ point of $\sgm$. Let $v=v_x$ be the corresponding valuation on $\sgf$.
Let $L_{z}$ be the completion of $L$ at $v_z$. Then the natural inclusion $L\subset L_{z}$ factors through $L\otimes_\sgf \sgf_v$ (which is the direct sum of the residue fields at the closed points of $\gy_{\sgf_v}$ over $y$).

We have the following commutative diagram (similar to \cite[p.167]{Kato}):
$$\xymatrix{H^n(\fy)\ar[r]\ar[d]^j& \prod\limits_{y \in  \gy^{(1)}}H^{n-1}(\kappa(y))\ar[d]^{j'}\\
\prod\limits_{v\in \valsgf} H^n(\fy_{v})\ar[r]& \prod\limits_v \prod \limits_{y'\in  \gy_{\sgf_v}^{(1)}} H^{n-1}(\kappa(y')).
}
$$

By assumption, $\alpha$ maps to zero over  $\fy_v$.
Hence the residue $\partial^n_{v_y}(\alpha)$ maps to zero in $\prod \limits_{y'\in  \gy_{\sgf_v}^{(1)}} H^{n-1}(\kappa(y')),$   hence is zero over any completion $\kappa(y)_z$ of $\kappa(y)$ as in the discussion above.
 Using that $n-1\geq 2$, by Corollary~\ref{mtol} we deduce that $\partial^n_{v_y}(\alpha)=0$. Hence by \cite[Part 1, II.7.9]{GMS:CohInv}
 and \cite[XII, Cor.~5.5(iii)]{SGA4-3},
the image $\alpha_y$ of  $\alpha$ in $H^n(\fy_w)$ satisfies
\[\alpha_y\in \im[H^n(\widehat {\mathcal O}_{\my, y})\hookrightarrow H^n(\fy_w)].\]
Let $\beta$ be the image of $\alpha_y$ under the natural isomorphism $H^n(\widehat {\mathcal O}_{\my, y})\stackrel{\sim}{\to} H^n(\kappa(y))$ given in
\cite[XII Cor 5.5(iii)]{SGA4-3}.
Similarly as above, since $\alpha$ maps to zero in $\prod\limits_{v\in \valsgf} H^n(\fy_{v})$ we deduce that $\beta$ satisfies
$$\beta\in \ker [H^n(\kappa(y))\to \prod_v \prod \limits_{y'\in  \gy_{\sgf_v}^{(1)}} H^{n}(\kappa(D_{y'}))].$$  
So $\beta=0$ by Corollary~\ref{mtol}
 again, and the image $\alpha_y$ of $\alpha$ in $H^{n}(\fy_w)$ is indeed zero.
\end{proof}

On the other hand, Proposition~\ref{comparisonHnr} does not hold in general for $n=2$; see Remark~\ref{Sha noninclus ex} below.

\begin{lemma}\label{comparisonHnrB}
Let $\fy$ be  the function field of a regular projective curve over a semi-global field $\sgf$.  Assume that  $[\sgf(\zeta_m) : \sgf]$ is prime to $m$. Then
$$\shasgf n(\fy)\supseteq \shafy n(\fy), \text{ for all }n\geq 2.$$
\end{lemma}

\begin{proof}
We first observe that it suffices to prove the lemma in the special case that $\zeta_m \in \sgf$.  To see this, assume that the result holds in this special case, and in the general case let $\alpha\in \shafy n(\fy) \subseteq H^n(\fy)$.
The restriction $\beta$ of $\alpha$ to $H^n(\fy(\zeta_m))$ lies in $\Sha^n_{\fy(\zeta_m)}(\fy(\zeta_m))$, and so by the special case it follows that $\beta \in \Sha^n_{\sgf(\zeta_m)}(\fy(\zeta_m))$.  That is, for every valuation $v' \in \Omega_{\sgf(\zeta_m)}$ lying over a valuation $v \in \Omega_{\sgf}$, the image $\beta_{v'}$ of $\beta$ in $H^n(\fy(\zeta_m)_{v'})$ is trivial.  Here $\beta_{v'}$ is the restriction of $\alpha_v$ to $H^n(\fy(\zeta_m)_{v'})$, where $\alpha_v$ is the image of $\alpha$ in $H^n(\fy_v)$.  But 
$[\fy(\zeta_m)_{v'} : \fy_v]$ is prime to $m$, so this restriction map is injective (by restriction-corestriction).  Hence each $\alpha_v$ is trivial; i.e., $\alpha$ lies in $\shasgf n(\fy)$, as asserted.

So for the remainder of the proof we assume that $\zeta_m \in F$, and we identify $H^n(\fy)$ with $H^n(\fy,\mu_m^{\otimes n})$.  Let $\alpha\in \shafy n(\fy) \subseteq H^n(\fy) = H^n(\fy,\mu_m^{\otimes n})$.  
Using the Bloch-Kato conjecture \cite{Voe11} (as in the proof of Theorem~\ref{sha mixed thm} above), we can write $\alpha=\sum_i (a_{i1})\cdot\ldots\cdot (a_{in})$, where  
$(a_{ij}) \in H^1(\fy,\mu_m)$, for some $a_{ij}\in \fy^\times$.
By definition, $\alpha$ maps to zero in $H^n(\fy_w,\mu_m^{\otimes n})$ for every $w\in \Omega_\fy$; and we wish to show that it maps to zero in $H^n(\fy_v,\mu_m^{\otimes n})$ for every $v \in \Omega_\sgf$.

Take $v\in \Omega_{\sgf}$, corresponding to a codimension one point $x$ on a regular model  $\sgm$ of $\sgf$. The local ring  ${\mathcal O}_{\sgm,x}$ is a discrete valuation ring. 
Let $\my_x\to {\mathcal O}_{\sgm,x}$ be a regular  model of $\fy$ over ${\mathcal O}_{\sgm,x}$ such that $\cup \Supp(a_{ij})$ is a union of regular curves with normal crossings on $\my_x$.  Such a model exists by \cite[page 193]{Lip75}, since $\dvr$ and hence ${\mathcal O}_{\sgm,x}$ is excellent.  Thus 
$\hat{\my}_x:= \my_x\times_{{\mathcal O}_{\sgm,x}} \hat{\mathcal O}_{\sgm,x}$ is also regular, and $\cup \Supp(a_{ij})$ is a union of regular curves with normal crossings on $\hat{\my}_x$ as well.
Here $E_v = \sgf_v(\gy)$, where $\gy$ is the generic fiber of $\my_x$ over $\sgf$; and this is 
a semi-global field over $\sgf_v$ that contains $E$.  

We need to prove that the image of $\alpha$ in $H^n(\sgf_v(\gy)) = H^n(\sgf_v(\gy), \mu_m^{\otimes n})$ is zero. 
Since $n \ge 2$, we may apply \cite[Theorem 3.2.3(i)]{HHK:LGGC} to the curve $\hat{\my}_x$ over the complete discrete valuation ring
$\hat{\mathcal O}_{\sgm,x}$.  By that result, it is enough to prove that for every point $P$ on the closed fiber $\hat Z$ of $\hat {\my}_x$, the image of $\alpha$ in $H^n(\sgf_v(\gy)_P, \mu_m^{\otimes n})$ is zero.
(Here $\sgf_v(\gy)_P$ is the fraction field of the complete local ring of 
$\hat {\my}_x$ at the point $P$.)

We may identify the closed fiber $Z$ of $\my_x$ with $\hat Z$, and thereby regard the above point $P \in \hat Z$ as a point of $Z \subset \my_x$.  Then $\sgf_v(\gy)_P$ is identified with
the fraction field $\fy_P$ of the complete local ring $\hat{\mathcal O}_{\my_x,P}$ of $\my_x$ at $P$. 
If $P$ is a codimension one point of $\my_x$, then it corresponds to some $w \in \Omega_\fy$.  Thus $\alpha\in \shafy n(\fy)$ becomes zero over $\fy_w = \fy_P$, as needed.  On the other hand, if $P$ is a codimension two point of $\my_x$, then take a system of local parameters $(\pi, \delta)$ at $P \in \my_x$, with $\pi, \delta$ each corresponding to an element of $\Omega_\fy$. 
By the above condition on $\cup \Supp(a_{ij})$, we may choose these parameters so that $\pi$ and $\delta$ are the only primes of $\hat{\mathcal O}_{\my_x,P}$
that divide all the $a_{ij}$.  Since $\alpha\in \shafy n(\fy)$ becomes zero over $\fy_{\pi} \subseteq (E_P)_{\pi}$ and over $\fy_{\delta} \subseteq (E_P)_{\delta}$, Lemma~\ref{la:coho cl van} implies that $\alpha$ becomes zero over $\fy_P = \sgf_v(\gy)_P$, completing the proof.
\end{proof}

Combining Proposition \ref{comparisonHnr} and Lemma \ref{comparisonHnrB},
we obtain:

\begin{prop} \label{Sha agree}
Let $\fy$ be the function field of a projective curve over a semi-global field $\sgf$.  
Assume that $[\sgf(\zeta_m) : \sgf]$ is prime to $m$.  Then
$\shasgf n(\fy)\supseteq\shafy n(\fy)$ for all $n \ge 2$,
and $\shasgf n(\fy)=\shafy n(\fy) \subseteq H^n_{\nr}(\fy)$ for all $n \ge 3$.
\end{prop}

\section{Higher Local-Global Principles} \label{results}

\subsection{Local-global principles with respect to valuations on $\sgf$} \label{all vals}

We now turn to the question of when local-global principles hold for a function field $\fy$ over a semi-global field $\sgf$; i.e., when the corresponding obstruction groups $\shasgf n(\fy)$ vanish.  (Here and below, we preserve Notation~\ref{notn}, for our fixed value of $m>1$.)  Using results of the previous section, we treat the case where $\fy = \sgf(\mathbb P^1_{\fy})$ in Theorem~\ref{LGP for line} below; and we consider the situation where the ground field $\dvf$ is a local field in 
Theorem~\ref{LGP over local field}.  First we state some preliminaries.

If $\sgm$ is a regular model of a semi-global field over $\dvr$, then the associated {\em reduction graph} was defined in \cite[Section~6]{HHK:H1}, encoding the configuration of intersections of the irreducible components of the closed fiber.  In the case that the model $\sgm$ has semi-stable reduction, this graph is homotopy equivalent to the dual graph defined in \cite[p.~86]{DM} (and is in fact its barycentric subdivision; see \cite[Remark~6.1]{HHK:H1}).

\begin{lemma} \label{non-tree cover}
If $\sgf$ is a semi-global field, then
 there is a finite separable extension $\sgf'$ of $\sgf$ having a regular model whose reduction graph is not a tree.
\end{lemma}

\begin{proof}
It suffices to prove that there exists a normal model $\sgm$ of $\sgf'$ with this property, since there is then a desingularization $\til\sgm \to \sgm$ by the main theorem in \cite{Lip78}; and its reduction graph, which maps onto that of $\sgm$, is then also not a tree.

First consider the case that $\sgf = \dvf(x)$, with a model 
$\sgm = \mbb P^1_\dvr$.  Pick an integer $r>1$ that is prime to the residue characteristic of $\dvr$.  Let $t$ be a uniformizer of $\dvr$, let $f,g \in \dvr[x]$ be irreducible monic polynomials whose loci in $\sgm$ are disjoint, and let $\sgm'$ be the normalization of $\sgm$ in the finite extension $\sgf' :=\sgf[z]/(z^n - (f^n - t)(g^n - t))$.  Then $\sgm'$ is a normal model of $\sgf'$, and the reduction graph of $\sgm'$ is not a tree.

For the general case, we may write $\sgf$ as a finite extension of a purely transcendental field $\sgf_0 :=\dvf(x)$.  Let $\sgf_1$ be the maximal separable subextension of $\sgf/\sgf_0$; thus $\sgf/\sgf_1$ is purely inseparable.  Since purely inseparable morphisms induce homeomorphisms on the underlying topological spaces, it suffices to prove the lemma for $\sgf_1$; and so we may assume that the finite extension $\sgf/\sgf_0$ is separable.  Here $\sgm_0 := \mbb P^1_\dvr$ is a regular model of $\dvf(x)$ over $\dvr$, and the normalization $\sgm$ of $\sgm_0$ in $\sgf$ is a normal model of $\sgf$, which is equipped with a finite generically separable morphism $\phi:\sgm \to \sgm_0$.
Pick two distinct closed points $P,Q$ on the closed fiber of $\sgm_0$ away from the point at infinity, over which $\sgm$ is regular, and which do not lie in the closure of any point in the generic fiber of $\sgm_0$ at which $\phi$ is branched.  The points $P,Q \in \mbb A^1_\dvr \subset \mbb P^1_\dvr$ are given by maximal ideals $(f,t)$ and $(g,t)$ in $\dvr[x]$, such that $f,g \in \dvr[x]$ are irreducible monic polynomials having disjoint loci in $\sgm_0$.  Choose $r>1$ relatively prime to the residue characteristic of $\dvr$.  Let $\sgm'_0 \to \sgm_0 = \mbb P^1_\dvr$ be the branched cover given in the previous paragraph, and let $\sgm'$ be the normalization of $\sgm'_0 \times_{\sgm_0} \sgm$.  The reduction graph of $\sgm'$ is not a tree, and  its function field $L$ is then as asserted.
\end{proof}

\begin{thm} \label{LGP for line}
Let $\sgf$ be a semi-global field over a complete discretely valued field $\dvf$.  Let $\fy = \sgf(t)$, the function field of $\mathbb P^1_\sgf$. 
Then:
\renewcommand{\theenumi}{\alph{enumi}}
\renewcommand{\labelenumi}{(\alph{enumi})}
\begin{enumerate}
\item\label{LGP line i}
$\shasgf 1(\fy)$ is trivial if and only if the reduction graph of a regular model $\sgm$ 
of $\sgf$ over $\dvr$ is a tree;
\item  \label{LGP line ii}
$\shasgf 2(\fy)$ is non-trivial;
\item \label{LGP line iii} 
$\shasgf n(\fy)$ is trivial for $n > 2$.
\end{enumerate}
\end{thm}

\begin{proof}
Assume first that $n>2$. Since $n \ge 3$,
$\shasgf n(\fy) \subseteq H^n_{\nr}(\fy)$
by Proposition~\ref{comparisonHnr}.
So by Proposition~\ref{ShaD}, each element of $\shasgf n(\fy)$ is in the image of $\shasgf n(\sgf) = \Sha^n_{\valsgf}(\sgf,\mu_m^{\otimes n-1})$.
By Corollary~\ref{mtol}, this last group is zero.  
This gives part~(\ref{LGP line iii}).

Assume now that $n=2$.  By Lemma~\ref{non-tree cover}, there is a finite extension $L/\sgf$ such that the reduction graph of a model of $L$ is not a tree. 
Let $y \in \gy = \mathbb P^1_\sgf$ be a closed point with  residue field $L$. Then in the right term of the sequence in Proposition~\ref{ShaD}, for $\kappa(y)=L$ we
have that $\shasgf 1(L)\neq 0$ (see \cite[Proposition 5.1, Corollary 5.3 and Corollary 6.5]{HHK:H1}).  Hence
$\shasgf 2(\fy) \neq 0$ as well, by Proposition~\ref{ShaD}, proving part~(\ref{LGP line ii}).

Finally, in  the case $n=1$, 
the last term in the exact sequence in Proposition~\ref{ShaD} is trivial since $\Sha^0$ is trivial.  So by that result,
$\shasgf n(\sgf) = \shasgf n(\fy)$.  The former group is trivial if and only if the reduction graph is a tree, by the results from \cite{HHK:H1} cited above in the $n=2$ case.  So assertion~(\ref{LGP line i}) follows.
\end{proof}

\begin{rem}
One has a similar result to part~(\ref{LGP line iii}) above: 
$\Sha^n_{\Omega}(\fy) := \Sha_{\Omega}^n(\fy, \mu_m^{\otimes n-1}) 
= \ker[H^n(\fy)\to \prod_{x\in\Omega}  H^n(\fy_x, \mu_m^{\otimes n-1})]$
is trivial for $n>2$, where $\Omega$ is the set of points on the closed fiber of $\sgm$; 
$\sgf_x = \frf(\wh {\ms O}_{\sgm, x})$ for $x \in \Omega$;
and $\fy_x = \frf(\fy \otimes_\sgf \sgf_x)$.
Namely, as before we have $\Sha^n_{\Omega}(\fy)\subseteq H^n_{\nr}(\fy)$, and so 
$\Sha^n_{\Omega}(\fy)$ is contained in the image of 
$\Sha^n_\Omega(\sgf) := \Sha_{\Omega}^n(\sgf, \mu_m^{\otimes n-1})$, which vanishes by Section~\ref{lgp pts}.
\end{rem}

\begin{rem} \label{Sha noninclus ex}
Theorem~\ref{LGP for line} shows that
Proposition~\ref{comparisonHnr} does not hold in general for $n=2$.  Namely, take $\dvf=\mathbb C((t))$, $\sgf = k(x)$, and $\fy=\sgf(y)$, with $\sgm = \mathbb P^1_{\mathbb C[[t]]}$ and $\my = \mathbb P^1_\sgm$. 
Then $\shasgf 2(\fy)$ is non-trivial by Theorem~\ref{LGP for line}(\ref{LGP line ii}). 
But we claim that
$\shamy 2(\fy)$ is trivial and hence so is $\shafy 2(\fy)$. To see the
triviality of $\shamy 2(\fy)$, 
recall that $\shamy 2(\fy) \subseteq H^2_{\nr}(\fy)$.
Since the unramified cohomology group $H^2_{\nr}(-)$ is a stable birational invariant, it
satisfies the conditions of specialization and homotopy (see \cite[p.~216]{CT80}),
and the natural map $H^2(k)\to H^2_{\nr}(\fy)$ is an isomorphism
by \cite[Th\'eor\`eme 1.5]{CT80}.  
(This conclusion can also be obtained by the Faddeev sequence and a diagram chase.)
But since $k = \mathbb C((t))$ is a $C_1$ field,
we have $H^2(k) = 0$, showing that $H^2_{\nr}(\fy) = 0$, and hence $\shamy 2(\fy)
= 0$ as claimed.
\end{rem}

Even if $\fy$ is not a purely transcendental extension of $\sgf$, we can still obtain  local-global principles if $\dvf$ is a local field. Theorem \ref{LGP over local field} below is a consequence of deep results of Saito-Sato, of de Jong and  of Gabber, and of Kerz-Saito. We thank J.-L.~Colliot-Th\'el\`ene for pointing out a theorem of Saito-Sato for the case $n=3$.

\begin{thm} \label{LGP over local field}
Let $\dvf$ be a non-archimedean local field, i.e.\ a complete discretely valued field with finite residue field.
Let $\sgf$ be a semi-global field over $\dvf$, 
and let $\fy$ be the function field of a regular projective curve $\gy$ over $\sgf$.  Then
$\shasgf n(\fy)=0, n \ge 3$.
\end{thm}

\begin{proof}
Assume first that $m$ is prime.
Let $\my$ be a normal proper model of $\fy$ over $\dvr$, the ring of integers of $\dvf$. Let $\my'\to \my$ be an alteration morphism of degree prime to $m$, with $\my'$ proper over $\dvr$ and regular, and with the special fiber a  simple normal crossings divisor (see the uniformization theorem of de Jong \cite[Theorem~4.1]{dJ96} and its refinement by Gabber \cite[Theorem~2.4]{IT14}; cf.\ also \cite[Theorem~1.2.5]{Tem17}).   
Let $\fy'$ be the field of functions of $\my'$.

For $n=3$, since $\my'$ is a regular scheme of total dimension $3$ that is proper over $\dvr$, with special fiber a simple normal crossings divisor, a theorem of Saito and Sato (\cite[Theorem 2.13]{SS10}, see also \cite[Th\'eor\`eme 3.16]{CT}) yields the vanishing of the group $H^3_{\nr}(\fy'/\my')$. 
(Here $H^3_{\nr}(\fy'/\my') = KH_3(\my',\Z/m\Z)$ in the notation of \cite{SS10}, since $\dim \my' =3$.)
Hence its subgroup $H^3_{\nr}(\fy')$ also vanishes. 

For $n=4$, since $\my'$ is a regular scheme that is proper over $\dvr$,
the Kato conjecture (proved by Kerz and Saito, \cite[Theorem 8.1]{KerzSaito}) 
asserts that the group $H^4_{\nr}(\fy'/\my')$ is zero. Hence so is its subgroup $H^4_{\nr}(\fy')$. 

For $n \ge 5$, the group $H^n(\fy')$ is zero by cohomological dimension.

 Since $[\fy':\fy]$ is prime to $m$, the map $H^n_{\nr}(\fy) \to H^n_{\nr}(\fy')$ is injective, and thus $H^n_{\nr}(\fy)=0$ for $n \ge 3$.  Hence so is $\shasgf n(\fy)$, which is contained in $H^n_{\nr}(\fy)$ by Proposition~\ref{comparisonHnr}. This concludes the case when $m$ is prime. The general case follows by the same argument as in the proof of Corollary \ref{mtol} (replacing $\shasgf n(\sgf)$ by $\shasgf n(\fy)$ in that proof).
\end{proof}

\begin{rem}
In the particular case that $\my$ is a smooth conic over $\sgf$, the vanishing of 
$\shasgf 3(\fy)$ also follows via \cite[Theorem 6.4]{PS16}
(which assumed that $\dvf$ is $p$-adic, but carries over to more general local fields $\dvf$ if $\sgf$ is the function field of a smooth $\dvf$-curve).
\end{rem}

\begin{rem}
Under the hypothesis that $[\sgf(\zeta_m) : \sgf]$ is prime to $m$, a restriction-corestriction argument shows that the above vanishing results about $\shasgf n(\fy)= \Sha^n_{\valsgf}(\fy,\mu_m^{\otimes (n-1)})$
also hold for $\Sha^n_{\valsgf}(\fy,\mu_m^{\otimes r})$ for any $r$.
The same is the case for the results in the next subsection about the vanishing of $\shasgfz n(\sgf)$
and $\shasgfz n(\fy)$.
\end{rem}

\begin{rem}
A related situation, over a global field rather than a local field, was considered by Jean-Louis Colliot-Th\'el\`ene and Bruno Kahn in \cite{CTK13}, in the case $n=3$.  
Section~8 there introduced a $\Sha$-type group of unramified classes  
$\Sha H^3_{\nr}(X, \mathbb Q/\mathbb Z(2)):=\ker[H^3_{\nr}(K(X), \mathbb Q/\mathbb Z(2)\to \prod_v H^3_{\nr}(K_v(X_v), \mathbb Q/\mathbb Z(2))]$, where
$X$ is a smooth projective variety over a global field $K$, and where the product is over all places $v$ of $K$.
Conjecture 8.9 in \cite{CTK13} asserts in particular that the group $\Sha H^3_{\nr}(X, \mathbb Q/\mathbb Z(2))$ vanishes if $X$ is a geometrically ruled surface over a function field in one variable over a finite field.
\end{rem}

\subsection{Local-global principles on the general fiber} \label{gen fiber vals}

The goal of this section is to investigate the case when the chosen set of discrete valuations does not include the valuations centered on the closed fiber.  Namely, we consider the set $\valsgf_{,0}$ of discrete valuations 
that are trivial on $\dvf$; i.e., those that correspond
to closed points of the generic fiber of $\sgm$.  We write   $\shasgfz n$  for the corresponding  $\Sha$-groups, and in particular
we write $\shasgfz n(\sgf) = \Sha^n_{\valsgf_{,0}}(\sgf,\mu_m^{\otimes n-1})$.  In the case of local-global principles for semi-global fields, this situation was considered in \cite{CH}, \cite{HSS}, and \cite{HS16}.

In this situation, the condition for a class $\alpha\in H^n(\fy)$ to be trivial over the fields $\fy_v$ for $v \in \valsgf_{,0}$ is much weaker than it is when we consider the set $\valsgf$.  In particular, if $\dvf$ is a local field, we show in Theorem~\ref{gen fiber Sha} that for $n>2$ the groups $\shasgfz n$ can be nontrivial even if $\fy$ is a purely transcendental field extension of $\sgf$  (contrary to what happens in Theorem~\ref{LGP for line}, where the larger set of discrete valuations was used).

\begin{lemma} \label{sha 1 agree}
Let $\sgf$ be a semi-global field over a non-archimedean local field $\dvf$.  Then $\shasgfz 1(\sgf)=\shasgf 1(\sgf)$ and $\shasgf n(\sgf) \subseteq \shasgfz n(\sgf)$ for $n>1$.
\end{lemma}

\begin{proof}
The containment $\shasgf n(\sgf) \subseteq \shasgfz n(\sgf)$ is immediate for all $n \ge 1$ since 
$\valsgf_{,0} \subset \valsgf$.  
It remains to show that $\shasgfz 1(\sgf)\subseteq\shasgf 1(\sgf)$.

Pick a regular model $\sgm$ of $\sgf$.  Then each element  
$f \in \shasgfz 1(\sgf) = \shasgfz 1(\sgf,\Z/m\Z)$ corresponds to a $\Z/m\Z$-Galois branched cover $\sgm' \to \sgm$ that is split at every closed point of the generic fiber of $\sgm$.  We claim that $\sgm' \to \sgm$ is also split over the generic point of each irreducible component of the reduced closed fiber.  Once this is shown, 
by \cite[Corollary~5.3]{HHK:H1}
it follows that $\sgm' \to \sgm$ is a split cover; i.e., it is split over every point of $\sgm$ except possibly the generic point.  Then by \cite[Proposition~8.2]{HHK:H1}, it follows that $f \in \shasgf 1(\sgf) = \shasgf 1(\sgf,\Z/m\Z)$, which will complete the proof.

To prove the claim, it suffices to show that it holds in the case that $m$ is a prime, since a $\Z/m\Z$-Galois branched cover is a tower of cyclic covers of prime order, and since a split cover is \'etale and therefore each step in the tower is also a regular model.  So for the remainder of the proof we assume that $m = \ell$, a prime (unequal to the residue characteristic of $\dvr$).

To show the claim with $m = \ell$, we may assume that $\zeta_\ell \in K$, since $[K(\zeta_\ell):K]$ has degree prime to $\ell$ and so it would suffice to show that the base change to $K(\zeta_\ell)$ satisfies the splitness condition.  With $\zeta_\ell \in K$, the branched cover $\sgm' \to \sgm$ is Kummer, and so is given generically by adjoining an $\ell$-th root of some element $a \in K$.  

Let $C$ be an irreducible component of the reduced closed fiber of $\sgm$.  Pick a regular closed point $P$ of the reduced closed fiber such that $P$ lies on $C$, and take a regular system of parameters $\sigma,\tau$ for $\ms O_{\sgm, P}$.  Since $P$ is a regular point of the reduced closed fiber, $C$ is the unique irreducible component of the closed fiber passing through $P$; and at most one of the two parameters vanishes at the generic point of $C$.  Thus at 
least one of these two parameters, say $\sigma$, instead vanishes along the closure $D$ in $\Spec(\ms O_{\sgm, P})$ of a closed point $z$ of the general fiber $\gsg$ of $\sgm$.  Since $\sgm' \to \sgm$ is split over every point of the general fiber, it follows that $a=\tau^r u$, where $u$ is a unit in $\ms O_{\sgm, P}$ and where we may choose $a$ so that $0 \le r < \ell$.

Now $D$ is the spectrum of a complete discrete valuation ring whose uniformizer is the reduction of $\tau$ modulo $\sigma$.  By hypothesis, $\sgm' \to \sgm$ splits over $z$.  Since the extension of $K$ corresponding to $\sgm' \to \sgm$ is obtained by adjoining an $\ell$-th root of $a$, it follows that $r=0$.  Thus 
$\sgm' \to \sgm$ is \'etale over $\Spec(\ms O_{\sgm, P})$, and hence is unramified over the generic point of $C$.  Moreover $\sgm' \to \sgm$ is split over $P$, since its restriction over 
$\Spec(D) \subset \Spec(\ms O_{\sgm, P}) \subset \sgm$ is an \'etale cover that is split over its generic point $z$.  Hence the restriction of $\sgm' \to \sgm$ to $C$ is a branched cover that is split over all but possibly finitely many closed points (viz., over all the closed points of $C$ where the closed fiber is regular).  Since $C$ is a curve over a finite field, Chebotarev's density theorem implies that this branched cover of $C$ is trivial.  This shows that $\sgm' \to \sgm$ splits over the generic point of every irreducible component of the closed fiber of $\sgm$, proving the claim.
\end{proof}

\begin{lemma} \label{ShaK}
Let $\sgf$ be the function field of a smooth projective curve $\gsg$ over a non-archimedean local field $\dvf$ of residue characteristic $p$. Let $\sgm$ be a regular proper model of $\gsg$ over the ring of integers of $\dvf$ and let $\Gamma$ be the reduction graph associated to the closed fiber of $\sgm$. Then
\renewcommand{\theenumi}{\alph{enumi}}
\renewcommand{\labelenumi}{(\alph{enumi})}
\begin{enumerate}
\item\label{ShaK i}
$\shasgfz 1(\sgf)=\Hom(\pi_1(\Gamma), \Z/m\Z)$, and this group is trivial if and only if $\Gamma$ is a tree.
\item\label{ShaK ii}
the groups $\shasgfz 1(\sgf)$ 
and $\shasgfz 3(\sgf)$ are dual; i.e., there is a perfect pairing to $\Q/\Z$.
\item\label{ShaK iii}
$\shasgfz 2(\sgf)=0$.
\end{enumerate}
\end{lemma}

\begin{proof}

First we prove (\ref{ShaK i}). 
By Lemma~\ref{sha 1 agree}, $\shasgfz 1(\sgf)=\shasgf 1(\sgf)$.  The elements of $\shasgf 1(\sgf)$ define $\Z/m\Z$-Galois split covers of $\sgm$; and conversely, each $\Z/m\Z$-Galois split cover of $\sgm$ corresponds to an element $H^1(\sgf,\Z/m\Z)$ that lies in $\shasgf 1(\sgf)$, by 
\cite[Proposition~8.2]{HHK:H1}.  By \cite[Corollary 6.4]{HHK:H1}, the set of split covers of $\sgm$ is classified by $\Hom(\pi_1(\Gamma), \Z/m\Z)$.  So assertion (\ref{ShaK i}) follows.

For part (\ref{ShaK ii}), 
the case where $k$ is a $p$-adic field follows from \cite[Theorem 4.4]{HS16}, which asserts in this situation that for any finite Galois module $A$ there is a perfect pairing between $\shasgfz i(\sgf,A)$ and $\shasgfz {4-i}(\sgf,A(-2))$; viz., we can take $i=3$ and $A=\mu_m^{\otimes 2}$.  The proof of that theorem carries over mutatis mutandis to the case of a smooth curve over an equal characteristic $p$ local field, provided that the torsion in $A$ is prime to $p$.  Since our choice of $m$ is assumed prime to the residue characteristic $p$, part (\ref{ShaK ii}) follows in that case too.

For (\ref{ShaK iii}), first note that $\shasgfz 2(\sgf) = \shasgfz 2(\sgf, \mu_m)$ consists of the $m$-torsion elements of $\Br(F)$ that are locally trivial with respect to $\valsgf_{,0}$.  Thus if 
$\alpha$ is a non-trivial element of $\shasgfz 2(\sgf, \mu_m)$, then a multiple of $\alpha$ is a non-trivial element of $\shasgfz 2(\sgf, \mu_\ell)$, for some prime $\ell$ that divides $m$.  Hence we are reduced to proving the result in the case that $m=\ell$ is prime.  In that situation, $[\sgf[\zeta_m]:\sgf] = [\dvf[\zeta_m]:\dvf]$ is prime to $m$; and so by a restriction-corestriction argument we may assume that $\zeta_m$ lies in $\dvf$, and hence in $\dvr$.  

Now let $\alpha\in \shasgfz 2(\sgf) = \shasgfz 2(\sgf, \mu_m)$.  We claim that 
$\partial_s(\alpha)=0$
for every codimension $1$ point $s\in \sgm^{(1)}$.  By \cite[Th\'eor\`eme~3.1, p.~98]{G}, the Brauer group of a regular curve $\sgm$ over the ring of integers of $\dvf$ is 
equal to that of its closed fiber and hence is trivial.  Thus the claim will imply that $\alpha=0$, and so will give (\ref{ShaK iii}).

Since the residue map at any closed point $x\in \gsg$ factors through $\sgf_v$, we deduce that $\alpha$ is unramified at $x$. Hence it only remains to check the claim that 
$\partial_s(\alpha) \in H^1(\kappa(s),\Z/m\Z)$ 
is trivial for $s$ a generic point of a component $C$ of the reduced closed fiber of $\sgm$.
But $\zeta_m \in \kappa(s)$, since $\zeta_m \in \dvr$.  So we may identify $H^1(\kappa(s),\Z/m\Z)$ with $H^1(\kappa(s),\mu_m) = \kappa(s)^\times/\kappa(s)^{\times m}$, and identify $\partial_s(\alpha)$ with the $m$-th power class $[f]$ of some $f \in \kappa(s)^\times$.

Let $P\in C$ be a closed point in the regular locus of the reduced closed fiber.
 The local ring $\mathcal {O}_{\sgm, P}$ is a regular local ring with maximal ideal generated by two local parameters. The zero locus of one of these parameters, which we call $u$, defines a regular connected closed subscheme  $Z\subset \sgm$ that is not contained in the closed fiber. Let $z$ be the generic point of $Z$. Consider the element $\beta=\alpha\cdot (u)$. Since $\alpha$ is trivial in $\sgf_z$ we deduce that $\beta$ is unramified at $Z$. Hence $\beta$ can  be ramified on $\sgm$ only along the closed fiber.  Then from the Gersten complex we obtain that 
 $$0=\partial_P{\partial_z(\alpha\cdot (u))}=-\partial_P\partial_s(\alpha\cdot (u))=-\partial_P((f)\cdot (\bar u)),$$
where $\bar u$ is the class of $u$ modulo the ideal of the reduced closed fiber. Hence, $\partial_P((f)\cdot (\bar u))=0$. Similarly, we have that $\partial_P(f)=\partial_P\partial_z(\alpha)=0$,  so that $f$ comes from $\wh {\ms{O}}_{C,P}^\times$. But since $\partial_P((f)\cdot (\bar u))$ is trivial,
 we deduce that the image of $f$ in the completion of $\kappa(C)$ at $P$ is an $m$-th power, for any $P\in C$ in the regular locus of the reduced closed fiber. By Chebotarev density, we deduce that $f$ is itself an $m$-power.  Thus $\partial_s(\alpha)=0$, proving the claim and hence~(\ref{ShaK iii}).
\end{proof}

\begin{rem}
\renewcommand{\theenumi}{\alph{enumi}}
\renewcommand{\labelenumi}{(\alph{enumi})}
\begin{enumerate}
\item
Part~(\ref{ShaK iii}) of the above lemma was previously shown in the $p$-adic case by 
using Lichtenbaum's duality pairing, in the second paragraph of the proof of part (2) of \cite[Proposition~3.4]{HS16}.
\item
In the above lemma, the assumption of smoothness (and not just regularity) was used only in part (\ref{ShaK ii}).  There, it was needed in order to carry over to the equal characteristic case the proof of \cite[Theorem 4.4]{HS16}, which relies on Artin-Verdier duality and hence on Poincar\'e duality.  As a result, smoothness is also used in the parts of the next theorem that rely on Lemma~\ref{ShaK}(\ref{ShaK ii}); viz., the case $n=4$ of part (\ref{gen fiber Sha i}), and part (\ref{gen fiber Sha ii}).
\end{enumerate}
\end{rem}

\begin{thm} \label{gen fiber Sha}
Let $\sgf$ be the function field of a smooth projective curve over a non-archimedean local field $\dvf$ of residue characteristic $p$ and let $\fy=\sgf(\mathbb P^1_\sgf)$ be a purely transcendental field extension of $\sgf$. Then
\renewcommand{\theenumi}{\alph{enumi}}
\renewcommand{\labelenumi}{(\alph{enumi})}
\begin{enumerate}
\item \label{gen fiber Sha i}
$\shasgfz n(\fy)\neq 0$ for $n=2,4$.
\item \label{gen fiber Sha ii}
The group $\shasgfz 3(\fy)$ is trivial if and only if the reduction graph of a regular model $\sgm$ of $\sgf$ over the ring of integers of $\dvf$ is  a tree.
\end{enumerate}
\end{thm}

\begin{proof}
We consider the exact sequence in Proposition~\ref{ShaD} with $\Omega=\valsgf_{,0}$.

For part~(\ref{gen fiber Sha i}), take $n=2$ in this exact sequence.  By Lemma~\ref{non-tree cover}, there is a finite separable extension $\sgf'$ of $\sgf$ having a model whose reduction graph is not a tree.  There is a point $y\in
\mathbb A^{1\,(1)}_\sgf$ with $\kappa(y)=\sgf'$; and the corresponding term in the right hand side of the exact sequence is non-zero by Lemma~\ref{ShaK}(\ref{ShaK i}).  
(Here we use that the valuations in $\Omega_{\sgf',0}$ are precisely the extensions to $\sgf'$ of the valuations in $\valsgf_{,0}$.)
Hence the middle term of the exact sequence, 
$\shasgfz 2(\fy)$, is also non-zero, proving the case $n=2$.

Similarly, taking $n=4$ in this exact sequence, Lemma~\ref{ShaK}(\ref{ShaK ii}) together with the previous case imply that the right hand term in the sequence is non-zero.  Hence $\shasgfz 4(\fy)$ is nonzero as well.

For part~(\ref{gen fiber Sha ii}), take $n=3$ in this exact sequence.  By Lemma~\ref{ShaK}(\ref{ShaK iii}), the right hand term in the sequence vanishes, and so the map 
$\shasgfz 3(\sgf)~\to~\shasgfz 3(\fy)$  
is an isomorphism. Hence the assertion follows from Lemma~\ref{ShaK}(\ref{ShaK i}),(\ref{ShaK ii}).
\end{proof}

We note that the case $n=2$ above can also be deduced from Theorem~\ref{LGP for line}(\ref{LGP line ii}) together with the fact that $\shasgf 2(\fy) \subseteq \shasgfz 2(\fy)$, where this inclusion follows from the containment $\valsgf_{,0} \subset \valsgf$.

\end{document}